\documentclass[a4paper,12pt,reqno]{amsart}
\usepackage{amssymb}
\usepackage{amsmath}
\usepackage{mathtools}
\usepackage{enumitem}
\usepackage{mathrsfs}
\usepackage[all]{xy}
\usepackage{mathtools}
\usepackage{hyperref}
\usepackage{cleveref}
\usepackage{url}
\usepackage{comment}

\usepackage[british]{babel}
\usepackage[margin=1.2in]{geometry}
\numberwithin{equation}{section} \numberwithin{figure}{section}

\DeclareMathOperator{\Gal}{Gal}

\DeclareMathOperator{\im}{Im}

\DeclareMathOperator{\ord}{ord} \DeclareMathOperator{\End}{End}

\DeclareMathOperator{\Frob}{Frob} 

\DeclareMathOperator{\Csplus}{C_{\text{s}}^+}

\newcommand{\OO}{\mathcal{O}}
\newcommand{\GL}{\textup{GL}}
\newcommand{\GSp}{\textup{GSp}}
\newcommand{\Sp}{\textup{Sp}}
\newcommand{\SL}{\textup{SL}}
\newcommand{\PGL}{\textup{PGL}}
\newcommand{\PGSp}{\textup{PGSp}}

\newcommand{\F}{\mathbb{F}}
\newcommand{\PP}{\mathbb{P}}
\newcommand\ZZ{\mathbb{Z}}

\newcommand{\Q}{\mathbb{Q}}

\newcommand{\fp}{\mathfrak{p}}

\newtheorem{lemma}{Lemma}

\newtheorem{theorem}[lemma]{Theorem}
\newtheorem{proposition}[lemma]{Proposition}
\newtheorem{corollary}[lemma]{Corollary}
\numberwithin{table}{section}

\theoremstyle{definition}
\newtheorem{example}[lemma]{Example}

\newtheorem{remark}[lemma]{Remark}

\numberwithin{lemma}{section}

\newcommand{\Mod}[1]{\ (\mathrm{mod}\ #1)}

\begin{document}

\title[Hasse surfaces]
{Examples of abelian surfaces failing the local-global principle for isogenies}

\author{\sc Barinder S. Banwait}
\email{barinder.s.banwait@gmail.com}


\begin{abstract}
We provide examples of abelian surfaces over number fields $K$ whose reductions at almost all good primes possess an isogeny of prime degree $\ell$ rational over the residue field, but which themselves do not admit a $K$-rational $\ell$-isogeny. This builds on work of Cullinan and Sutherland. When $K=\Q$, we identify certain weight-$2$ newforms $f$ with quadratic Fourier coefficients whose associated modular abelian surfaces $A_f$ exhibit such a failure of a local-global principle for isogenies.
\end{abstract}

\maketitle

\section{Introduction}

Let $A$ be an abelian variety over a number field $K$, and $\ell$ a prime number. If $A$ admits a $K$-rational $\ell$-isogeny, then necessarily, at every prime $\fp$ of good reduction not dividing $\ell$, the reduction $\tilde{A}_\fp$ over $\F_\fp$ also admits an $\ell$-isogeny, rational over $\F_\fp$. One may ask the converse question:

\begin{center}
\emph{
	If $A$ admits a rational $\ell$-isogeny locally at every prime of good reduction away from $\ell$, must $A$ admit a $K$-rational $\ell$-isogeny?
}
\end{center}

If the answer to this question for a given pair $(A/K,\ell)$ is `No', we refer to $\ell$ as an exceptional prime for $A$, and refer to $A$ as a \emph{Hasse at $\ell$ variety over $K$}. We think of Hasse at $\ell$ varieties as being counterexamples to a local-global principle for $\ell$-isogenies.

This problem has been studied extensively in the case where $A$ is an elliptic curve, starting with the work of Sutherland \cite{Drew} who provided a characterisation of Hasse curves in terms of the \emph{projective mod-$\ell$ Galois image} (whose definition we recall in \Cref{sec:prelims}), and found all such counterexamples in the case when $K = \Q$ (of which there is only one up to isomorphism over $\overline{\Q}$).

Cullinan \cite{cullinan2012symplectic} initiated the study of this question in the case of $\dim A = 2$, by identifying the subgroups of $\GSp_4(\F_\ell)$ that the mod-$\ell$ Galois image of a Hasse at $\ell$ variety must be isomorphic to, and remarked that, while his classification could be used to generate Hasse surfaces over arbitrary base fields, it ``would be interesting to create ``natural'' examples of such surfaces''.

In this paper we provide the first examples of Hasse at $\ell$ surfaces that are simple over $\Q$, by studying the abelian varieties $A_f$ associated to weight~$2$ newforms $f$ via the Eichler-Shimura construction:

\begin{example}\label{example:cm_hasse}
Consider the weight two newform of level $\Gamma_1(189)$, Nebentypus the non-primitive Dirichlet character modulo $189$ of conductor $21$, sending the two generators $29$ and $136$ of the group $(\ZZ/189\ZZ)^\times$ to $-1$ and $\zeta_6^5$, where $\zeta_6 := e^{2\pi i/6}$ respectively, whose first few Fourier coefficients are as follows:

\[ f(z) = q + (-2 + 2\zeta_6)q^4 + (-1 + 3\zeta_6)q^7 + O(q^{10}). \]

Then $A_f$ is a Hasse at 7 abelian surface over $\Q$. This $f$ has label \href{https://www.lmfdb.org/ModularForm/GL2/Q/holomorphic/189/2/p/a/}{189.2.p.a} in the \href{https://www.lmfdb.org/}{LMFDB} \cite{lmfdb}.
\end{example}

This $f$ is a CM newform, having complex multiplication by the field $\Q(\sqrt{-3})$, and as such $A_f$ necessarily decomposes over $\overline{\Q}$ as the square of a CM elliptic curve. Our next example provides an instance of an absolutely simple Hasse surface.

\begin{example}\label{example:abs_simple_hasse}
Consider the weight two newform of level $\Gamma_0(7938)$ with Fourier coefficient field $\Q(\sqrt{2})$, whose first few coefficients are as follows ($\beta = \sqrt{2}$):

\[ f(z) = q - q^2 + q^4 - q^8 - 9\beta q^{11} + O(q^{12}). \]

Then $A_f$ is an absolutely simple Hasse at 7 abelian surface over $\Q$. This $f$ has label \href{https://www.lmfdb.org/ModularForm/GL2/Q/holomorphic/7938/2/a/bj/}{7938.2.a.bj} in the LMFDB.
\end{example}

Although the $f$ in this example does not have CM, one may show that it is congruent to a CM newform modulo $7$. We show that this is to be expected:

\begin{theorem}\label{thm:cm_congruence}
Let $f$ be a weight $2$ newform such that the corresponding modular abelian variety $A_f$ is Hasse at some prime $\ell$ which splits completely in the ring of integers of the Hecke eigenvalue field of $f$. Then $f$ is congruent modulo $\ell$ to a newform with complex multiplication.
\end{theorem}

The structure of the paper is as follows. In \Cref{sec:prelims} we survey previous and related work on this question, including Sutherland's group-theoretic reformulation of Hasse at $\ell$ varieties. \Cref{sec:decomposable_abelian_surfaces} studies the modular abelian varieties $A_f$ indicated above, yielding sufficient conditions on $f$ to ensure that $A_f$ is Hasse. Sections \ref{sec:find_examples_using_code} and \ref{sec:abs_simple_hasse} explain the algorithmic ingredients required to find examples of newforms satisfying the sufficient conditions, including the two examples given above. Finally in \Cref{sec:cm_congruence} we prove \Cref{thm:cm_congruence}.

\section{Background and Preliminaries}\label{sec:prelims}

For an abelian variety $A$ over a number field $K$, the absolute Galois group $G_K := \Gal(\overline{K}/K)$ acts on the $\ell$-torsion subgroup $A(\overline{K})[\ell]$, yielding the mod-$\ell$ representation

\[ \bar{\rho}_{A,\ell} : G_K \to \GL_{2d}(\F_\ell),\]

whose image $G_{A,\ell} := \im \bar{\rho}_{A,\ell}$ is well-defined up to conjugacy; we refer to $G_{A,\ell}$ as \emph{the mod-$\ell$ image of $A$}. We let $H_{A,\ell} := G_{A,\ell}$ modulo scalars, which we refer to as \emph{the projective mod-$\ell$ image of $A$}, viewed as a subgroup of $\PGL_{2d}(\F_\ell)$. If $A$ admits a polarisation of degree coprime to $\ell$, then the symplectic property of the Weil pairing on $A[\ell]$ ensures that $G_{A,\ell}$ is  contained in $\GSp_{2d}(\F_\ell)$, and consequently that $H_{A,\ell} \subseteq \PGSp_{2d}(\F_\ell)$. Henceforth we will assume that $A$ is principally polarised. 

By an $\ell$-isogeny $\phi : A \to A'$ of principally polarised abelian varieties of dimension $d$ defined over a field $k$ with char($k$) $\neq \ell$ we mean a surjective morphism with kernel isomorphic to $\ZZ/\ell\ZZ$. We note that these isogenies are \emph{not} compatible with the principal polarisations of $A$ and $A'$, since this kernel is not a maximal isotropic subgroup of $A[\ell]$ with respect to the $\ell$-Weil pairing. To consider isogenies that \emph{are} compatible with the polarisations, one would need to consider certain isogenies with kernel isomorphic to $(\ZZ/\ell\ZZ)^d$, often denoted as $(\ell,\cdots,\ell)$-isogenies (see e.g. \cite{costello2020supersingular}). One may well formulate a local-global question for such isotropic isogenies, and the results in \cite{orr2017compatibility} are likely to be relevant here; but we do not address this problem in the present paper.

Sutherland's characterisation of Hasse curves mentioned in the Introduction is expressed in terms of the canonical faithful action of $H_{A,\ell}$ on the projective space $\PP^{2d-1}(\F_\ell)$. Following our previous paper \cite{BC13}, given a subgroup $H$ of $\PGSp_{2d}(\F_\ell)$, we say that $H$ is \emph{Hasse} if its action on $\PP^{2d-1}(\F_\ell)$ satisfies the following two properties:

\begin{itemize}
	\item
	every element $h \in H$ fixes a point in $\PP^{2d-1}(\F_\ell)$;
	\item
	there is no point in $\PP^{2d-1}(\F_\ell)$ fixed by the whole of $H$.
\end{itemize}

We also refer to a subgroup $G$ of $\GSp_{2d}(\F_\ell)$ as Hasse if its image modulo scalars is Hasse.

The following result is then used by Sutherland in the case of $\dim A = 1$: the details of the general case are entirely analogous, and may be found spelled out in \cite{BanThesis}, Section 2.2:

\begin{proposition}[Sutherland]\label{prop:group_theoretic_reformulation}
An abelian variety $A/K$ is Hasse at $\ell$ if and only if $H_{A,\ell}$ is Hasse.
\end{proposition}

In the case $\dim A = 1$, it is easy to show that no subgroup of $\PGL_{2}(\F_2)$ is Hasse, so for elliptic curves the prime $2$ is never an exceptional prime. For an odd prime $\ell$, define $\ell^\ast := +\ell$ if $\ell \equiv 1 \Mod{4}$, and $\ell^\ast := -\ell$ otherwise.

Sutherland provides necessary conditions for an elliptic curve $E$ over a number field $K$ to be Hasse at an odd prime $\ell$, under the assumption that $\sqrt{\ell^\ast} \notin K$, which is equivalent to the determinant of the projective representation $\PP\bar{\rho}_{E,\ell}$ being surjective (see Lemma~2.1 in \cite{BC13}). These conditions were shown to be sufficient in Section 7 of \cite{BC13}. In the following Proposition, by $D_{2n}$ we mean the dihedral group of order $2n$.

\begin{proposition}[\cite{Drew}, \cite{BC13}]\label{prop:hasse_elliptic}
Let $\ell$ be an odd prime, $K$ a number field, and assume that $\sqrt{\ell^\ast} \notin K$. Then an elliptic curve $E$ over $K$ is Hasse at $\ell$ if and only if the following hold:
\begin{enumerate}
	\item
	the projective mod-$\ell$ image of $E$ is isomorphic to $D_{2n}$, where $n > 1$ is an odd divisor of $(\ell-1)/2$;
	\item
	$\ell \equiv 3 \Mod{4}$;
	\item
	the mod-$\ell$ image of $E$ is contained in the normaliser of a split Cartan subgroup of $\GL_2(\F_\ell)$;
	\item
	$E$ obtains a rational $\ell$-isogeny over $K(\sqrt{\ell^\ast})$.
\end{enumerate}
\end{proposition}

\begin{remark}
For the converse of the above Proposition, only conditions (1) and (2) are required; together these imply conditions (3) and (4).
\end{remark}

\begin{remark}
The case of $\sqrt{\ell^\ast} \in K$ was dealt with independently by \cite{BanThesis} and \cite{AnniThesis} (see also \cite{anni2014local}).
\end{remark}

The property of an elliptic curve $E$ being Hasse at some prime $\ell$ depends only on $j(E)$, provided $j(E) \notin \left\{0,1728\right\}$. Sutherland therefore defines an \emph{exceptional pair} to be a pair $(\ell, j_0)$ of a prime $\ell$ and an element $j_0 \neq 0, 1728$ of a number field $K$ such that there exists a Hasse at $\ell$ curve over $K$ of $j$-invariant $j_0$.

Sutherland moreover shows, in the proof of Theorem $2$ in \cite{Drew}, that a Hasse curve cannot have CM if $\ell > 7$; therefore, specialising now to $K = \Q$, elliptic curves with level structure given by (3) above arise as non-trivial points on the modular curve $X_s(\ell)$ (the trivial points being the cusps and CM points). That such points exist only for $\ell \in \left\{2,3,5,7,13\right\}$ follows from the work of Bilu, Parent and Rebolledo \cite{BPR}, although Sutherland was able to deduce the following remarkable result using the earlier work of Parent \cite{parent2005towards}, as well as an explicit study of the modular curve $X_{D_6}(7)$ and its rational points.

\begin{theorem}[Sutherland]
The only exceptional pair for $\Q$ is

	\[ \left(7,\frac{2268945}{128}\right).\]

\end{theorem}

The analogue of \Cref{prop:hasse_elliptic} providing precisely which subgroups of $\PGSp_4(\F_\ell)$ are Hasse was given by Cullinan \cite{cullinan2012symplectic}. Given a subgroup $H \subseteq \PGSp_4(\F_\ell)$, let $\pi^{-1}(H)$ denote the pullback of $H$ to $\GSp_4(\F_\ell)$.

\begin{theorem}[Cullinan]
A subgroup $H \subseteq \PGSp_4(\F_\ell)$ is Hasse if and only if $\pi^{-1}(H) \cap \Sp_4(\F_\ell)$ is isomorphic to one of the groups in \Cref{tab:cullinan}.
\end{theorem}

\begin{table}[htp]
\begin{center}
\begin{tabular}{|c|c|c|}
\hline
Type & Group & Condition\\
\hline
$\mathcal{C}_2$ & $D_{(\ell-1)/2} \wr S_2$ & None\\
& $\Csplus$ & $\ell \equiv 1$(4)\\
& $(\ell-1)/2.\SL_2(\F_3).2$ & $\ell \equiv 1$(24)\\
& $(\ell-1)/2.\GL_2(\F_3).2$ & $\ell \equiv 1$(24)\\
& $(\ell-1)/2.\widehat{S_4}.2$ & $\ell \equiv 1$(24)\\
& $(\ell-1)/2.\SL_2(\F_5).2$ & $\ell \equiv 1$(60)\\ %[1ex]
& $\SL_2(\F_3) \wr S_2$ & $\ell \equiv 1$(48)\\
& $\widehat{S_4} \wr S_2$ & $\ell \equiv 1$(48)\\
& $\SL_2(\F_5) \wr S_2$ & $\ell \equiv 1$(120)\\
\hline
$\mathcal{C}_6$ & $2^{1+4}_{-}.O_4^{-}(2)$ & $\ell \equiv 1$(120)\\
& $2^{1+4}_{-}.3$ & $\ell \equiv 5$(24)\\
& $2^{1+4}_{-}.5$ & $\ell \equiv 5$(40)\\
& $2^{1+4}_{-}.S_3$ & $\ell \equiv 5$(24)\\
\hline
$\mathcal{S}$ & $2.S_6$ & $\ell \equiv 1$(120)\\
& $\SL_2(\F_5)$ & $\ell \equiv 1$(30)\\
& $\SL_2(\F_3)$ & $\ell \equiv 1$(24)\\
\hline
\end{tabular}
\vspace{0.3cm}
\caption{\label{tab:cullinan}Hasse subgroups of $\PGSp_4(\F_\ell)$. See \cite{cullinan2012symplectic} for the group-theoretic notation used in this table.}
\end{center}
\end{table}

At this point we may readily engineer Hasse surfaces over arbitrary number fields. For example, suppose we would like to construct an abelian surface $A$ whose mod-$\ell$ image satisfies $G_{A,\ell} \cap \Sp_4(\F_\ell) \cong \SL_2(\F_5)$ for some prime $\ell \equiv 1$ (mod 30); by \Cref{tab:cullinan}, this would give a Hasse surface. We would first take an abelian surface over $\Q$ with absolute endomorphism ring isomorphic to $\ZZ$; a quick search in the LMFDB yields the genus~$2$ curve \href{https://www.lmfdb.org/Genus2Curve/Q/249/a/249/1}{249.a.249.1}:

\[ \mathcal{C} : y^2 + (x^3 + 1)y = x^2 + x,\]

whose Jacobian variety $A$ has conductor $249$ and $\End_{\overline{\Q}}(A) \cong \ZZ$. Serre's Open Image Theorem, which also holds for abelian surfaces with absolute endomorphism ring $\ZZ$ \cite{hall2011open} ensures that, for all sufficiently large primes $\ell$, we have $G_{A,\ell} \cong \GSp_4(\F_\ell)$. Moreover, Dieulefait \cite{dieulefait2002explicit} provides an algorithm to determine a bound on the primes of non-maximal image. This algorithm has recently been implemented \cite{galreps} in Sage \cite{sagemath} at an ICERM workshop funded by the Simons collaboration, and for this $A$ we find that any prime $\ell \geq 11$ ensures maximal image. Choose such an $\ell$ which is congruent to $1$ (mod $30$), e.g. $\ell = 31$. We finally base-change $A$ to force $G_{A,\ell} \cap \Sp_4(\F_\ell) \cong \SL_2(\F_5)$, using the Galois correspondence.

\begin{example}
The Jacobian variety of the curve $\mathcal{C}$ above is a Hasse at $31$ surface over the number field $K$ such that $\Gal(\Q(A[31])/K) \cong \SL_2(\F_5)$.
\end{example}

\begin{remark}
We indicate here other work on this subject. These local-global type questions for abelian varieties go back to Katz in 1980 \cite{katz1980galois}, who studied the analogous local-global question for rational torsion points; for elliptic curves this goes even further back to the exercises in I-1.1 and IV-1.3 in Serre's seminal book \cite{serre1968abelian}. Etropolski \cite{etropolski2015local} considers a local-global question for arbitrary subgroups of $\GL_2(\F_\ell)$, and Vogt \cite{vogt2020local} generalises the prime-degree-isogeny problem to composite degree isogenies. Very recently Mayle \cite{mayle} bounds by $\frac{3}{4}$ the density of prime ideals for elliptic curves $E/K$ which do not satisfy either of the ``everywhere-local'' conditions for torsion or isogenies, and Cullinan, Kenney and Voight study a probabilistic version of the torsion local-global principle for elliptic curves \cite{cullinan2020probabilistic}.
\end{remark}

\section{Split modular abelian surfaces which are Hasse} \label{sec:decomposable_abelian_surfaces}

The example constructed in the last section raises the question of whether there are Hasse surfaces over $\Q$, pre-empting this somewhat contrived base-change method. In approaching this question, we establish the following lemma.

\begin{lemma}\label{main_result}
Let $A$ be an abelian surface over a number field $K$ whose mod-$\ell$ Galois image $G_{A,\ell}$ is contained in the direct sum of two subgroups $G, G'$ of $\GL_2(\F_\ell)$:

\[ G_{A,\ell} \subseteq \begin{pmatrix}
G & 0 \\
0 & G' 
\end{pmatrix}. \]

If one of $\left\{G,G'\right\}$ is Hasse, and the other is not contained in a Borel subgroup, then $A$ is a Hasse at $\ell$ surface over $K$.
\end{lemma}

\begin{proof}
Let $H_{A,\ell}, H, H'$ respectively denote the images of $G_{A,\ell}, G, G'$ modulo scalar matrices. By \Cref{prop:group_theoretic_reformulation}, we need to establish that $H_{A,\ell}$ is a Hasse subgroup. Since the mod-$\ell$ Galois representation in this case decomposes as a direct sum of two subrepresentations, we write $V, V'$ such that $A[\ell] = V \oplus V'$.

We first show that $H_{A,\ell}$ does not fix a point in $\PP(A[\ell])$. If it did, then that point lifts to a point $w \in A[\ell]$. We may write $w = v \oplus v'$, with $v \in V$, $v' \in V'$. Since at least one of $v, v'$ must be non-zero, we suppose that $v$ is non-zero. Then $H$ must fix the image of $v$ in $\PP(V)$, which is not allowed under the hypotheses on $\left\{G,G'\right\}$.

Without loss of generality we suppose that $H$ is Hasse. Each element of $H_{A,\ell}$ may be written as $y = \begin{pmatrix}
h & 0 \\
0 & h' 
\end{pmatrix}$ for $h \in H$, $h' \in H'$. Since $h$ fixes a point in $V$, $y$ fixes the same point; thus every element of $H_{A,\ell}$ fixes a point.
\end{proof}

An immediate corollary provides an example of a Hasse surface over $\Q$, using Sutherland's $j$-invariant defined above:

\begin{corollary}
Let $E/\Q$ be any elliptic curve with $j$-invariant $\frac{2268945}{128}$. Then the abelian surface $E^2$ is Hasse at $7$ over $\Q$.\qed
\end{corollary}

This prompts the question of whether there exist \emph{simple} Hasse surfaces over $\Q$. We provide an affirmative answer to this question by restricting to the class of \emph{modular abelian surfaces} over $\Q$, whose definition we now recall.

Let $f$ be a weight~$2$ cuspidal newform of level $\Gamma_1(N)$ for some $N > 1$, with Fourier coefficient field $K_f$, a number field whose ring of integers we will denote as $\OO_f$. In the course of constructing the $\ell$-adic Galois representations of $f$, Shimura (Theorem 7.14 in \cite{shimura1971introduction}) defined the abelian variety $A_f$ associated to $f$, whose dimension is $[K_f:\Q]$. It is a theorem of Ribet (Corollary 4.2 in \cite{ribet1980twists}) that these abelian varieties are simple over $\Q$, and that $K_f$ is the full algebra of endomorphisms of $A_f$ which are defined over $\Q$. In this paper we refer to these varieties $A_f$ as \emph{modular abelian varieties}, and in the case where $[K_f:\Q] = 2$, we call them \emph{modular abelian surfaces}. (The reader is warned however that the adjective \emph{modular} is used by different authors throughout the literature to mean different things.)

Furthermore, these varieties are of \textbf{$\GL_2$-type}: the $\ell$-adic Tate module, for each $\ell$, splits as a direct sum

\[ T_\ell A_f = \bigoplus_{\lambda | \ell}T_{f,\lambda}, \]

where each $T_{f,\lambda}$ is a free module of rank~$2$ over the $\lambda$-adic completion $\OO_{f,\lambda}$ of $\OO_f$. (See Exercise 9.5.2 in \cite{diamond2005first}; to obtain the integrality one may need to replace $T_{f,\lambda}$ with a similar representation, as explained in the discussion immediately preceding Definition 9.6.10 in \emph{loc. cit.}. This decomposition is also explained in Section~$2$ of \cite{ribet1977galois}). This formula allows us to consider the $\ell$-adic representation $T_\ell A_f$ as a direct sum of the $2$-dimensional $\lambda$-adic representations associated to $f$.

Consider the case in which $K_f$ is a quadratic field, and $(\ell) = \lambda\lambda'$ splits in $\OO_f$. By taking the reduction mod $\ell$ of the above formula, we obtain a splitting

\[ A_f[\ell] = \overline{T}_{f,\lambda} \oplus \overline{T}_{f,\lambda'}\]

of the $4$-dimensional $G_\Q$-representation $A_f[\ell]$ as a sum of two $2$-dimensional representations, all considered as representations over $\F_\ell$. Thus $G_{A_f,\ell}$ is contained in the block sum of two subgroups $G,G'$ of $\GL_2(\F_\ell)$:

\[ G_{A_f,\ell} \subseteq \begin{pmatrix}
G & 0 \\
0 & G' 
\end{pmatrix}. \]

We choose $G$ and $G'$ minimally; i.e., $G$ is the image of $G_\Q$ acting on $\overline{T}_{f,\lambda}$, and $G'$ the image of $G_\Q$ acting on $\overline{T}_{f,\lambda'}$. We denote by $H$ and $H'$ the corresponding projective images, as subgroups of $\PGL_2(\F_\ell)$.

We may therefore state sufficient conditions for a modular abelian surface $A_f$ to be Hasse, as a corollary of \Cref{prop:hasse_elliptic} and \Cref{main_result} above:

\begin{corollary}\label{cor:suff_conds_for_hasse}
Let $f$ be a weight $2$ newform of level $\Gamma_1(N)$ with Fourier coefficient field $K_f$. Suppose:

\begin{itemize}
	\item
	$K_f$ is a quadratic field;
	\item
	$\ell \geq 7$ is a prime congruent to $3 \Mod{4}$ which splits in $\OO_f$ as $(\ell) = \lambda\lambda'$;
	\item
	among the projective mod-$\lambda$ and mod-$\lambda'$ images, one is isomorphic to $D_{2n}$, where $n > 1$ is an odd divisor of $\frac{l-1}{2}$, and the other is not contained in a Borel subgroup.
\end{itemize}

Then $A_f$ is Hasse at $\ell$ over $\Q$.\qed

\end{corollary}

\begin{remark}
We do not deal with the case of $\ell$ remaining inert or ramifying in $\OO_f$ in this paper. This would likely involve a group-theoretic investigation of the Hasse subgroups of $\PGL_2(\F_{\ell^n})$.
\end{remark}

In the next section we apply an algorithm of Anni \cite{AnniThesis} which determines when a weight $k$ newform has projective dihedral image, in order to find an $f$ satisfying the assumptions in the above corollary.

We end this section with a result which gives sufficient conditions on $f$ to ensure that both the mod-$\lambda$ and mod-$\lambda'$ images are isomorphic. This enables us, in certain situations, to consider the image for only one of the prime ideals above $\ell$. Recall that the Fourier coefficient field of a newform $f$ is either totally real, or a CM field.

\begin{proposition}\label{prop:iso_image}
Let $f$ be a weight $2$ newform of level $\Gamma_1(N)$ and Fourier coefficient field $K_f$. Suppose that $K_f$ is an imaginary quadratic field, and $\ell$ splits in $\mathcal{O}_f$ as $(\ell) = \lambda\lambda'$. Then the projective mod-$\lambda$ and mod-$\lambda'$ images are isomorphic.
\end{proposition}

\begin{proof}
Denoting by $\epsilon$ the Nebentypus of $f$, observe that we have the following relation:

\[ \bar{f} = f \otimes \epsilon^{-1}, \]

where the bar denotes complex conjugation (see e.g. \S~1 or the proof of Proposition~3.2 in \cite{ribet1977galois}). Since $K_f$ is imaginary, this gives a non-trivial element in the group of inner twists of $f$, which sends $f$ to its Galois conjugate, swaps $\lambda$ and $\lambda'$, and induces an isomorphism $\rho_{f,\lambda'} \cong \rho_{\bar{f},\lambda}$. We conclude by observing that $f$ and $f \otimes \epsilon^{-1}$ have isomorphic projective mod-$\lambda$ image.
\end{proof}

\begin{remark}
In the case that $f$ does not have CM, the assumption in the above proposition that $K_f$ is a CM field is equivalent to the assumption that the Nebentypus of $f$ is not trivial (c.f. Example~3.7 in \cite{ribet1980twists}).
\end{remark}

\section{Constructing examples using Anni's thesis}\label{sec:find_examples_using_code}

Section 10.1 of \cite{AnniThesis} describes an algorithm (Algorithm 10.1.3 in \emph{loc. cit.}) to determine whether or not a weight $k$ newform has projective dihedral image modulo a prime ideal $\lambda$ of the ring of integers $\OO_f$ of $K_f$. The main idea can be encapsulated in the following:

\begin{proposition}[Anni, Ribet, Serre]
Let $f$ be a weight $k$ newform of level $N$, and let $\rho$ be the mod-$\lambda$ Galois representation associated to $f$. Assume that $\rho$ is irreducible. Then the following are equivalent:

\begin{enumerate}
	\item
	$\rho$ has projective dihedral image;
	\item
	there exists a quadratic character $\alpha$ of modulus $q$ such that $\alpha \otimes \rho \cong \rho$, where $q$ is the product of all primes dividing $N$ such that their square divides $N$;
	\item
	there exists a quadratic field $K$, and characters $\chi, \chi'$ on $G_K$, such that the restriction of $\rho$ to $G_K$ is reducible:
	\[ \rho|_{G_K} = \chi \oplus \chi'. \]
\end{enumerate}

Moreover, if these hold, then the order of the dihedral group is $2n$, where $n$ is the order of $\chi^{-1}\chi'$.
\end{proposition}

We refer to the relevant results in the literature for more details: in chronological order, Proposition 4.4 and Theorem 4.5 in \cite{ribet1977galois}, Section 7 of \cite{serre1977modular}, and Section 10.1 of \cite{AnniThesis}.

Anni's algorithm then consists in checking whether one of the finitely many Dirichlet characters as described in (2) above satisfies $\alpha \otimes \rho \cong \rho$, noting that only the primes up to the Sturm bound need to be checked.

At this point, if Anni's algorithm yields a quadratic character for such a newform $f$, then either it has projective dihedral image, \emph{or} the representation is reducible, which would mean it has cyclic image. This reducible case is equivalent to $f$ being congruent mod-$\ell$ to an Eisenstein series of the same weight and level, which may be checked by computing the finitely many normalised Eisenstein series.

If the representation is indeed dihedral, then we compute the characteristic polynomials of Frobenius at several rational primes to determine its order.

We implemented this algorithm in Sage (\verb|find_dihedral.sage| in \cite{dihedralnewforms2020}), and ran it on all two-dimensional weight-two newforms $f$ (those with $[K_f:\Q] = 2)$ of level $\leq 189$, for which the prime $7$ splits in $\OO_f$. 

The results obtained are summarised in \Cref{tab:dihedral-newforms}. We found that all of the forms had CM, and that the projective images in all of these cases were isomorphic for each of the prime ideals above $7$, as is necessarily the case in light of \Cref{prop:iso_image}. To save space in the table, we note here that the Fourier coefficient field of all of these newforms is the quadratic field $\Q(\sqrt{-3})$, and remind the reader that with $D_n$ we mean the dihedral group \emph{of order $n$} (and not $2n$).

\begin{table}[htp]
\begin{center}
\begin{tabular}{|c|c|c|c|c|}
\hline
LMFDB Label & CM field & $q$-expansion & $\PP\rho(G_\Q)$\\
\hline
49.2.c.a & $\Q(\sqrt{-7})$ & $q - \zeta_6q^2 + (1 - \zeta_6)q^4 - 3q^8 + 3\zeta_6q^9 + O(q^{10})$ & $C_3$\\
\hline
63.2.e.a & $\Q(\sqrt{-3})$ & $q + 2\zeta_6q^4 + (1 - 3\zeta_6)q^7 + O(q^{10})$ & $D_4$\\
\hline
81.2.c.a & $\Q(\sqrt{-3})$ & $q + 2\zeta_6q^4 + (1 - \zeta_6)q^7 + O(q^{10})$ & $D_{12}$\\
\hline
117.2.g.a & $\Q(\sqrt{-3})$ & $q + 2\zeta_6q^4 + \zeta_6q^7 + O(q^{10})$ & $D_{12}$\\
\hline
117.2.q.b & $\Q(\sqrt{-3})$ & $q - 2\zeta_6q^4 + (6 - 3\zeta_6)q^7 + O(q^{10})$ & $D_{12}$\\
\hline
189.2.c.a & $\Q(\sqrt{-3})$ & $q + 2q^4 + (-1 + 3\zeta_6)q^7 + O(q^{10})$ & $D_6$\\
\hline
189.2.e.b & $\Q(\sqrt{-3})$ & $q + (2 - 2\zeta_6)q^4 + (1 - 3\zeta_6)q^7 + O(q^{10})$ & $D_{12}$\\
\hline
189.2.p.a & $\Q(\sqrt{-3})$ & $q + (-2 + 2\zeta_6)q^4 + (-1 + 3\zeta_6)q^7 + O(q^{10})$ & $D_6$\\
\hline
\end{tabular}
\vspace{0.3cm}
\caption{\label{tab:dihedral-newforms}Newforms arising as output from Anni's algorithm. The two forms with projective image $D_6$ yield Hasse surfaces over $\Q$ at $7$.}
\end{center}
\end{table}

For the newforms in the table whose level is prime to $7$, we verified the irreducibility of the mod-$\lambda$ Galois representation with Corollary~2.2 of \cite{dieulefait2001newforms}: if it was reducible, then there would exist a Dirichlet character $\chi$ of conductor dividing the level and valued in $\F_7^\times$ such that, for all primes $p$ away from the level, we would have

\[ a_p \equiv \chi(p) + p\frac{\epsilon(p)}{\chi(p)} \Mod{\lambda},\]

where $\epsilon$ is the Nebentypus of $f$. Since there are only finitely many such $\chi$, we can test all possible candidates, and find that none of them satisfy all of these congruences, whence the representation must be irreducible.

The last example in the above table is given in \Cref{example:cm_hasse}. Since it is a CM form, the corresponding abelian variety $A_f$ decomposes over $\overline{\Q}$ as the square of a CM elliptic curve $E$. We may quickly glean further information about $E$ from the \href{https://www.lmfdb.org/ModularForm/GL2/Q/holomorphic/189/2/p/a/}{homepage} of this form; in particular, from the ``Related objects'' section we find that the decomposition occurs over the field $\Q(\sqrt{-3})$, and that $E$ is a curve with $j_E = 0$, whose mod-$7$ Galois image is a split Cartan subgroup (and not its normaliser). By Sutherland's work \Cref{prop:hasse_elliptic}, we may conclude that $E$ is not a Hasse curve.

\section{Finding absolutely simple Hasse Modular Abelian surfaces}\label{sec:abs_simple_hasse}

We first collect some facts about absolutely simple modular abelian varieties from the literature.

\begin{proposition}[Cremona, Jordan, Ribet] \label{lem:abs_endo}
Let $f$ be a weight $2$ newform of level $\Gamma_1(N)$ such that the corresponding modular abelian surface $A_f$ is Hasse at some prime $\ell$ which splits completely in $\OO_f$. Assume that $f$ is not a CM newform.

\begin{itemize}
	\item
	If $f$ does not have inner twists, then $A_f$ is absolutely simple.
	\item
	If $f$ does have inner twists, then $A_f$ is absolutely simple if and only if $\End_{\overline{\Q}}^0(A_f)$ is an indefinite quaternion division algebra with centre $\Q$ of degree $4$ over $\Q$. Moreover, if this holds, then this algebra is realised over a totally complex field, $A_f$ has potential good reduction everywhere, and for every prime $p$ dividing $N$, we have $\ord_p(N) \geq 2$.
\end{itemize}
\end{proposition}

\begin{proof}
Write $\mathcal{X} = \End_{\overline{\Q}}^0(A_f)$. We have the following facts:

\begin{enumerate}
	\item
	the centre of $\mathcal{X}$ is a subfield $F$ of $K_f$, and  $\mathcal{X} \cong M_n(\cdot)$, where $\cdot$ is either $F$, or else an indefinite quaternion division algebra over $F$ of dimension $t^2$ over $F$, where $t$ is the Schur index of $\mathcal{X}$ (Proposition 5.2 in \cite{Ribet2004});
	\item
	the degree of $\mathcal{X}$ over $\Q$ is $2[K_f:F]$ (Theorem 5.1 in \cite{ribet1980twists});
	\item
	$F$ is a totally real number field, and $\Gal(K_f/F)$ is the group of inner twists of $f$ (Corollary 5.4 in \cite{Ribet2004});
	\item
	$[K_f:F] = nt$ (Proposition 5.2 in \cite{Ribet2004}).
\end{enumerate}

If $f$ does not have inner twists, then $\Gal(K_f/F)$ is trivial, and so $n=1$; i.e., $A_f$ is absolutely simple.

If $f$ does have inner twists, then we have $2 = nt$, so $n = 1 \Leftrightarrow t=2$; i.e., $A_f$ is absolutely simple if and only if $\mathcal{X}$ is an indefinite quaternion division algebra over $F$ of degree $4$ over $\Q$.

The statements about the endomorphisms being realised over a totally complex field, and $A_f$ having potential good reduction everywhere, follow from the observation that $A_f/K$, when base-changed to the field $K$ over which all endomorphisms are defined, satisfies the definition of a \emph{fake elliptic curve}, and thus follow from the known properties of these objects; see e.g. Section 4 of \cite{halukandsamir}, who attribute this to Jordan (Section 3 in \cite{jordan1986points}).

The statement about the valuations of primes dividing $N$ follows from Theorem 3 in \cite{cremona1992abelian}.
\end{proof}

One could in principle run the algorithm explained in \Cref{sec:find_examples_using_code} on all non-CM newforms with no non-trivial inner twists to furnish an example. However, for each such level $N$, the implementation \verb|find_dihedral.sage| in \emph{loc. cit.} first constructs the entire space of newforms $S_2(\Gamma_1(N))^{new}$, and thereafter takes only those of dimension~$2$; as such, it is very inefficient.

We therefore implemented a faster approach in \verb|find_simple_dihedral_with_api.sage| in \emph{loc. cit.}, which refactors the main algorithms in \verb|find_dihedral.sage| to take as input not \texttt{Newform} objects, but rather lists of Fourier coefficients at prime values of newforms. A list of the newform labels to be checked is generated from the LMFDB, with the following parameters:

\begin{itemize}
	\item Dimension $2$;
	\item No CM;
	\item Inner twist count $1$.
\end{itemize}

For each label in this list, the Fourier coefficients $a_p$ for prime $p$ are obtained with a call to the \href{http://www.lmfdb.org/api/}{LMFDB API}. Running the refactored algorithm to find dihedral newforms on all 15,838 candidate newforms in the LMFDB takes about half an hour on an old laptop.

The results obtained are summarised in \Cref{tab:abs-simple-dihedral-newforms}. All the forms have Fourier coefficient field $\Q(\sqrt{2})$; we write $\beta = \sqrt{2}$. The projective images are verified as before, computing the orders of root quotients of the characteristic polynomials of Frobenius at several primes. This in particular allows one to rule out reducibility of the representation, by showing that the distribution of orders is inconsistent with a cyclic group. Unlike in \Cref{sec:find_examples_using_code}, the projective images at the two prime ideals above $7$ are not isomorphic; we provide the dihedral image, which occurs at the prime ideal given in the table. For the other prime ideal not given, where the algorithm returns that it does not have dihedral image, one readily finds a prime $p$ such that the characteristic polynomial of $\Frob_p$ is irreducible over $\F_7$, and hence the image is not contained in a Borel subgroup, which is sufficient for our purposes from \Cref{cor:suff_conds_for_hasse}.

\begin{table}[htp]
\begin{center}
\begin{tabular}{|c|c|c|c|c|}
\hline
LMFDB Label & $q$-expansion & $\PP\rho(G_\Q)$ & Prime ideal \\
\hline
7938.2.a.bj & $q - q^2 + q^4 - q^8 - 9\beta q^{11} + O(q^{12})$ & $D_6$ & $(1 - 2\beta)$\\
\hline
7938.2.a.bk & $q - q^2 + q^4 - q^8 + 3\beta q^{11} + O(q^{12})$ & $D_6$ & $(1 + 2\beta)$\\
\hline
7938.2.a.bp & $q + q^2 + q^4 + q^8 + 9\beta q^{11} + O(q^{12})$ & $D_6$ & $(1 - 2\beta)$\\
\hline
7938.2.a.bq & $q + q^2 + q^4 + q^8 + 9\beta q^{11} + O(q^{12})$ & $D_6$ & $(1 - 2\beta)$\\
\hline
9099.2.a.e & $q - 2q^4 + (-3 - \beta)q^5 + (-2 + 2\beta)q^7 + O(q^{12})$ & $D_{12}$ & $(1 - 2\beta)$\\
\hline
9099.2.a.g & $q - 2q^4 + (3 + \beta)q^5 + (-2 + 2\beta)q^7 + O(q^{12})$ & $D_{12}$ & $(1 - 2\beta)$\\
\hline
\end{tabular}
\vspace{0.3cm}
\caption{\label{tab:abs-simple-dihedral-newforms}Newforms arising as output from Anni's algorithm using the LMFDB API. The three forms with projective image $D_6$ yield absolutely simple Hasse surfaces over $\Q$ at $7$.}
\end{center}
\end{table}

The first example in the above table is \Cref{example:abs_simple_hasse} from the Introduction, which yields a Hasse at $7$ surface by \Cref{cor:suff_conds_for_hasse}.

\section{Modular Hasse surfaces are congruent to CM newforms} \label{sec:cm_congruence}

In this section we prove \Cref{thm:cm_congruence}.

Let $f \in S_2(\Gamma_1(N))$ be a newform, and $\ell$ a prime which splits completely in the ring of integers $\mathcal{O}_f$ of the Fourier coefficient field $K_f$. By assumption that $A_f$ is Hasse, there exists a prime ideal $\lambda | \ell$ such that the projective image of $\overline{\rho}_{f,\lambda}$ is a Hasse subgroup of $\PGL_2(\F_\ell)$. Therefore, by \Cref{prop:hasse_elliptic}, we have that $\im \PP \overline{\rho}_{f,\lambda}$ is a dihedral group. 

Henceforth, for ease of notation, write $\overline{\rho}$ for $\overline{\rho}_{f,\lambda}$. Since $\det \overline{\rho}$ is surjective in $\F_\ell^\times$, we have that $\det \PP \overline{\rho}$ is surjective in $\left\{\pm 1\right\}$. The kernel of $\det$ is an index-$2$ subgroup of a dihedral group of order $2n$ with $n$ odd, and therefore is cyclic of order $n$. We thus obtain that the kernel of the composition

\[ G_\Q \xrightarrow{\PP \overline{\rho}} D_{2n} \longrightarrow D_{2n}/C_n \longrightarrow \left\{\pm1\right\} \]

corresponds to the imaginary quadratic field $\Q(\sqrt{-\ell})$.

We may now apply Th\'{e}or\`{e}me 1.1 of \cite{billerey2018representations} to obtain the existence of a CM newform $g$ such that $\overline{\rho}$ is isomorphic to the mod-$\lambda'$ reduction of the $\lambda'$-adic $G_\Q$-representation $\rho_{g,\lambda'}$, for some prime ideal $\lambda'$ lying over $\ell$ in the Fourier coefficient field of $g$ (which need not be the same as that of $f$). Moreover, from the proof of Corollaire 1.3 in \emph{loc. cit.}, we have that the weight of $g$ is $2$. This yields the desired congruence.\qed

\begin{remark}
Theorem A in \cite{orr_skorobogatov_2018} tells us that there are only finitely many $\overline{\Q}$-isomorphism classes of abelian surfaces over $\Q$ with complex multiplication. There are therefore only finitely many $\overline{\Q}$-isomorphism classes of Hasse modular abelian surfaces with CM. Since the field of complex multiplication in this case must be an imaginary quadratic field of class number 1 or 2, there are only finitely many such. Note that Gonz\'{a}lez (Theorem 3.2 in \cite{gonzalez2011}) has enumerated the possible pairs $(\End_{\overline{\Q}}^0(A_f), \End_{\Q}^0(A_f))$, for $A_f$ a two-dimensional modular abelian surface with complex multiplication; there are 83 such pairs.
\end{remark}

\section{Acknowledgements}
	This work was supported by a grant from the Simons Foundation (546235)
	for the collaboration `Arithmetic Geometry, Number Theory, and
	Computation', through a workshop held virtually at ICERM in June 2020. I am deeply indebted to the organisers of that workshop for extending to me an invitation for participation. I particularly thank John Voight for publicly wondering ``what happens for abelian surfaces'' after Jacob Mayle's talk, which inspired me to return to this subject after a seven-year hiatus.

	I thank Alex Bartel for comments on an earlier version of this manuscript; Nicolas Billerey for a correspondence which clarified issues surrounding congruences between CM and non-CM newforms; Peter Bruin for a correspondence about reducible Galois representations of newforms, for comments on an earlier draft of the manuscript, and for verifying the projective images of some weight-$2$ newforms that arose as output to Anni's algorithm - the current check on reducibility via searching for Eisenstein series congruences is from Sage code that he provided to me; John Cremona for extensive comments on an earlier version of the manuscript; David Loeffler for a correspondence which identified the role of non-trivial inner twists in \Cref{prop:iso_image}; Nicolas Mascot for explaining how to detect dihedral image via traces of Frobenius, suggesting the algorithms in Anni's thesis, and for corrections to an earlier version of the manuscript; Martin Orr for insightful examples about lifting mod-$p$ dihedral representations to characteristic zero, as well as for comments on and corrections to an earlier version of the manuscript; Samir Siksek for suggestions and ideas for further development; Andrew Sutherland for explaining how to use the LMFDB API to obtain Fourier coefficients of modular forms; and John Voight for questions about the CM examples from an earlier version of the manuscript.

	I thank John Cremona for giving me his copy of \emph{Modular Curves and Abelian Varieties} on my last day at Warwick as his PhD student, and Jonny Evans for giving me his copy of \emph{A First Course in Modular Forms} as I was leaving Cambridge. Both proved to be essential in the course of this work.

	I wish to express my sincere gratitude to all those who have provided open access to otherwise prohibitively expensive material.

\bibliographystyle{alpha}
\bibliography{/home/barinder/Documents/database.bib}{}
\end{document}